\documentclass[a4paper,11pt,makeidx]{amsart} 
\usepackage{amscd}             
\usepackage{xypic}  
\usepackage{amssymb}
\usepackage{amsthm}
\usepackage{epsf} 
\usepackage[T1]{fontenc}            
\makeindex
\newtheorem{thm}{Theorem}[section]   

\newtheorem{exa}[thm]{Example}
\newtheorem{lemma}[thm]{Lemma}
\newtheorem{prop}[thm]{Proposition}
\newtheorem{example}[thm]{Example}
\newtheorem{defn}[thm]{Definition}

\newtheorem{rem}[thm]{Remark}

\def\i{^{-1}}

\def\ker{\operatorname{ker}}

\def\max{\operatorname{max}}

\def\c1{\operatorname{c_1}}
\def\c2{\operatorname{c_2}}

\def\PP{{\mathbf P}}

\def\G{{\mathcal G}}

\def\O{{\mathcal O}}

\def\E{{\mathcal E}}

\def\F{{\mathcal F}}

\def\V{{\mathcal V}}
\def\W{{\mathcal W}}
\def\x{\times}                   

\def\+{\oplus}                   
\def\*{\otimes}                  
\def\khpil{\rightarrow}

\def\Hom{\operatorname{Hom}}

\def\ker{\operatorname{Ker}}

\def\Fq{\mathbb{F}_{q}}
\def\Kx{K_{X}}
\def\Ox{O_{X}}
\def\sPrin{\underline{\mathrm{Prin}}}
\def\Prin{\mathrm{Prin}}
\def\sRat{\underline{\mathrm{Rat}}}
\def\Rat{\mathrm{Rat}}
\def\x{\mathbf{x}}
\def\y{\mathbf{y}}
\def\e{\mathbf{e}}

\def\strs{{\mathcal O}_{X}}
\def\cans{{\mathcal K}_{X}}

\begin{document}

\title{Decoding of scroll codes}

\author{George H.\ Hitching and Trygve Johnsen}

  
\address{Inst.\ of Algebraic Geometry\\
Leibniz University, Hannover\\ Welfengarten 1\\ 30167 Hannover\\ Germany\\
and\\
Dept. of Mathematics\\ 
University of Bergen\\ Johs.\ Brunsgt 12\\ N-5008 Bergen\\ Norway}

\email{hitching@math.uni-hannover.de and johnsen@math.uib.no}

\begin{abstract} We define and study a class of codes obtained from scrolls over curves of any genus over finite fields. These codes generalize Goppa codes in a natural way, and the orthogonal complements of these codes belong to the same class. We show how syndromes of error vectors correspond to certain vector bundle extensions, and how decoding is associated to finding destabilizing bundles. \end{abstract}

\maketitle

\section{Introduction} \label{intro}
In \cite{J}, the second author interpreted the syndrome space for traditional Goppa codes $C(D,G)$ for divisors $D,G$ on an algebraic curve $X$ as a projective space $\PP = \PP Ext(H,\Ox)^*$ of isomorphism classes of line bundle extensions. It is well known that an extension of line bundles 
\[ 0 \khpil \O_X \khpil W  \khpil H \khpil 0 \]
is classified by its cohomology class $\delta(W) \in H^{1}(X , \Hom(H, \Ox)) \cong Ext(H, \Ox)$, the null element of $Ext(H,\O_X)$ corresponding to the class of the trivial extension
\[ 0 \khpil \O_X \khpil \O_X \oplus H \khpil H \khpil 0.\]
The curve $X$ is embedded in $\PP$ in the following way: If $x \in X$, then $x$ corresponds to the class of an extension which is the kernel of the map
\[ Ext(H,\O_X) \khpil Ext(H,\O_X(x)). \]
In that case the middle term $E$ has a quotient bundle $\O_X(x)$. Moreover, in general, the class of an extension $W$ is in the kernel of the map
\[ Ext(H,\O_X) \khpil Ext(H,\O_X(A)), \]
for an effective divisor $A$ if and only if it is contained in $Span(A)$ (defined in a standard way) after embedding $X$ in $\PP$ as described. In that case $W$ has a quotient bundle $\Ox(A).$ In this way $\PP$ is stratified into secant strata of the embedded $X$ according to the $s$-invariant of the middle terms $W$ of the extensions appearing. The process of error location then corresponds, for given syndrome (= class of extension $W$), to find the right divisor $A'$ linearly equivalent to $A$ such that $\O_X(A)$ is a quotient bundle of the middle term $W$. It is interesting to observe that syndromes of errors which are designed-correctable (that is, the number of errors is at most $[\frac{d-1}{2}]$, where $d$ is the designed minimum distance $deg(D-G)$) are precisely the ones corresponding to unstable extensions. Recall that an extension is unstable if and only if $s(W)=2deg(A)-deg(W) < 0$ for a line bundle $A$ of minimal degree such that $\O_X(A)$ is a quotient bundle of the middle term $W$; equivalently, if $W$ contains a line subbundle of degree greater than $\frac{1}{2} deg(W)$. This viewpoint has been utilized and studied through a series of papers; see \cite{BC}, \cite{Co1}, \cite{Co2}.
\par
In this paper we will replace the divisor or line bundle $G$ on the curve $X$ with a locally free sheaf $\E$ (or vector bundle $E$) of arbitrary positive rank $r$ and apply a similar construction. This gives rise to a scroll $\PP E$ and various scroll codes, some of which have been studied in several papers. See \cite{Ha}, \cite{L}, and \cite{Na}. We show how syndromes and decoding can be interpreted in terms of vector bundle extensions for a particular class of such codes.
\par
Here is a summary of the article. Firstly, we recall some facts about scrolls and vector bundles which will be needed. In $\S$\ref{matrix}, we define ``SAGS codes'', a type of evaluation code which generalizes Goppa's SAG codes to scrolls. In particular, these have the property that their dual codes can again be interpreted as evaluation codes. In $\S$\ref{geomext}, we recall or prove some facts about the geometry of vector bundle extensions, and in $\S$\ref{geomcorr} we apply this to decoding and error correction on SAGS codes. In the final section, we make brief remarks about the applicability of these results to scroll codes which are not necessarily evaluation codes.
\par
An important tool is the use of bundle-valued principal parts to define the codes. We believe this makes transparent the connection between syndromes, bundles and geometry.\\
\\
\textbf{Acknowledgements:} The first author is supported by the Deutsche Forschungsgemeinschaft Schwerpunktprogramm ``Globale Methoden in der komplexen Geometrie''. He also thanks the University of Bergen for financial support and hospitality.

\section{Scrolls and vector bundles} \label{sins}

In this section we introduce the objects with which we will be working. Firstly, we fix some notation.\\
\par
We denote vector bundles over the curve $X$ with Roman letters $E$, $W$, $\Ox$, $\Kx$ etc.\ and their sheaves of sections with the corresponding script letters $\E$, $\W$, $\strs$, $\cans$ etc. If $V$ is a vector space, then $\PP V$ is the projective space of codimension one linear subspaces in $V$. Similarly, for a vector bundle $E \to X$ we define $\PP E$ to be the scroll whose fiber at $x \in X$ is the projective space of codimension one linear subspaces of $E|_x$. If $\Upsilon$ is a line bundle over some variety $Y$, we write $|\Upsilon|$ for the projective space $\PP H^{0}(Y, \Upsilon)$. If $|\Upsilon|$ is nonempty, we have a natural map $Y \dashrightarrow |\Upsilon|$.
\par
If $V$ is a vector space and $g \in V^*$ a nonzero linear form, we denote $\langle g \rangle$ the line in $V^*$ spanned by $g$, and also the point in $\PP V$ defined by $g$. We also use this notation for points of projectivized vector bundles.
\par
Any vector bundle $E \to X$ gives rise to a short exact sequence of $\strs$-modules
\[ 0 \to \E \to \sRat(E) \to \sPrin(E) \to 0 \]
where $\sRat(E)$ is the sheaf of rational sections of $E$ and $\sPrin(E)$ the sheaf of principal parts\footnote{Note that this is a different object from the ``principal part sheaf'' ${\mathcal P}^{k}_{X}(\E)$ considered by for example Laksov \cite{La}.} with values in $E$. Taking global sections, we obtain
\begin{equation} 0 \to H^{0}(X,E) \to \Rat(E) \to \Prin(E) \to H^{1}(X, E) \to 0. \label{cohomseq} \end{equation}
We denote $\overline{\alpha}$ the principal part of a global rational section $\alpha$ of $E$, and we write $[p]$ for the cohomology class of a principal part $p \in \Prin(E)$. See for example \cite{K} for further information.

\begin{defn} \label{ratscro}
Let $\E$ be a locally free sheaf of rank $r \ge 1$ on a curve $X$, chosen in such a way that the linear system 
$\Upsilon=\O_{\PP E}(1)$ on the corresponding $\PP^{r-1}$-bundle $\PP E$
over $X$ is very ample, and $h^1(\Upsilon)=0$. We map $\PP E$ into $\PP^{k-1}$ with the complete linear system $H^0 (\Upsilon)$. The image $T$ is by definition a smooth
scroll, and isomorphic to $\PP E$.
\par
In particular, if $X=\PP^1$, then $\E=\O_{\PP^1}(e_1) \+  \cdots \+\O_{\PP^1}(e_r)$, with $e_1 \geq  \ldots  \geq e_r
\geq 1$ and $\deg E=f=e_1+ \cdots +e_r \geq 2$. In this case $k=f+r,$ and
the image $T$ is by definition a rational normal 
scroll of type ${\bf e} = (e_1, \ldots ,e_r)$. 
\end{defn}
\begin{rem}
{\rm If the locally free sheaf $\E$ satisfies certain stability conditions, then the dimension $k$ is equal to $\deg(\E)+r(1-g)$ also for non-rational curves (twisting $E$ with a large enough multiple of the line bundle corresponding to a fiber $F$ if necessary). In general (see \cite{Na}, Proposition 2.1),
\[ h^0 \left( \O_{\PP E}(b_1) \otimes \O_{\PP E}(b_2 F) \right) = { b_1+r-1 \choose\ r-1 } (\mu b_1 +b_2 + 1-g )\]
in this case, where $\mu (\E)$ is the slope $\frac{\deg \E}{r}$ of $E$. Here we only study the case $b_1=1$. From the proof of this result it also follows that $h^1(X,E)=0$ under these stability conditions, and this gives $h^1(\PP E,\Upsilon)=0.$}
\end{rem}
In the next section, we will describe some codes which can be produced from these objects.
%
%
%
%

\section{Matrix description}
\label{matrix}
In this section we will define a generalization of the strongly algebraic geometric (SAG) codes considered in \cite{J}.

\subsection{Strongly algebraic geometric scroll codes}

Let $C$ be a code over a finite field $\Fq$ defined as follows. Start with a scroll $\PP E$ over a curve $X$, which is embedded in $\PP^{k-1}$ as described above. Suppose $\gamma$ is the number of $\Fq$-rational points on $X$; if $X = \PP^1$ then $\gamma = q+1$. Then we recall that $T$ contains
\begin{equation} \label{number}
n=\gamma(q^{r-1}+q^{r-2}+ \dots +q+1)
\end{equation}
points over $\Fq$. Choose $s$ of the $\gamma$ fibers of $\PP E$ over $X$, and in each fiber we pick at least $r$ points, such that these points span the fiber.
Altogether we have then chosen $v$ points $P_1, \dots, P_v,$ and
$sr \le v \le n$.
Let $\Upsilon$ be the linear system on $\PP E$ described above, and look at the map
$\phi \colon H^{0}(\PP E, \Upsilon) \khpil (\Fq)^v$ defined by $\phi (f) = (f(P_1), \dots, f(P_v)).$ The code $C$ is the image of $\phi$.
\par
Let $M$ be the divisor on $\PP E$ corresponding to the $s$ fibers spanned by the $P_i$, so $M$ is numerically (linearly if $X=\PP^1$) equivalent to $sF$ on $\PP E$, where 
$F$ is the class of a fiber. Recall that $\Upsilon$ is the bundle associated to the hyperplane
system $\O_{\PP E}(1)$. Look at the exact sequence of sheaves
$$ 0 \khpil \Upsilon(-M) \khpil \Upsilon \khpil \frac{\Upsilon}{\Upsilon(-M)} \khpil 0. $$
This induces an exact cohomology sequence
$$ 0 \khpil H^{0}(\Upsilon(-M)) \khpil H^{0}(\Upsilon) \khpil (\Fq)^{sr} \khpil H^1(\Upsilon(-M)) \khpil H^1(\Upsilon) \khpil 0. $$
In turn this induces a sequence of maps
$$ 0 \khpil H^{0}(\Upsilon(-M)) \khpil H^{0}(\Upsilon) \khpil (\Fq)^{sr} \khpil (\Fq)^{v}, $$
where each function on the union of the $s$ chosen fibers is evaluated at
the $v$ points by the last map of the sequence. We denote this map by $g$. Of course we claim no exactness
of the last sequence at $(\Fq)^{sr}$. We see from this that we can regard the linear code $C$ as the image of the quotient space $\frac{ H^0(\Upsilon)}{ H^0(\Upsilon(-M))}.$ In a special case considered by many authors one picks all $\Fq$-rational points in all fibers, so $s=\gamma$ and we pick $q^{r-1}+ \dots +q+1$ points in each fiber, and then $v=n=\gamma(q^{r-1}+ \dots +q+1).$
The last two sequences above are simplified if $H^0(\Upsilon(-M))=0$. For $X = \PP^1$ this happens if $s \ge e_1+1$, and such an $s$ can be chosen if 
$q \ge e_1$.

We now look at a special case:
\\
Here we pick instead exactly $r$ points in each of the $s$ fibers, and we also pick them such that they span the fibers. Write $D$ for the sum of the points of $X$ over which the divisor $M$ on $\PP E$ is supported. Clearly this is of the form $x_{1} + \cdots + x_s$ for distinct $\Fq$-rational points $x_{i} \in X$. For each $i=1 , \ldots , s$ we will denote the points in the fiber over $x_i$ by $P_{i,1}, \ldots , P_{i,r}$.
\par
Then $v=sr$, the map $g$ described above is an isomorphism of vector spaces, and we may identify the spaces $(\Fq)^{sr}$ and $(\Fq)^{v}$ of the last sequence, and regard the map $H^0(\Upsilon) \khpil (\Fq)^{sr}=(\Fq)^{v}$ of the long exact cohomology sequence as an evaluation map in the $v=sr$ points.
\par
Now the cohomology  $H^0(\Upsilon)$ and  $H^1(\Upsilon(-M))$ can be identified with cohomology spaces of bundles on $X$. We have
\[ H^0( \PP E , \Upsilon )= H^0(X,E) \]
and 
\begin{align*} H^1( \PP E ,\Upsilon(-M)) &= H^1(X, E \otimes \pi_*(\O_{\PP E}(-M))) \\ 
&= H^1(X, E \otimes \O(-D)) \\
&= H^0(X, \Kx(D) \otimes E^*)^{*} \hbox{ by Serre duality}\\
&= H^0( \PP E_{1} ,\Upsilon_1)^* \end{align*}
where $\Upsilon_1$ is a suitable line bundle on a scroll $\PP E_1$. Here $E_1= \Kx(D) \otimes E^*$, and $\Upsilon_1$ is $\O_{\PP E_1} (1)$ for this locally free sheaf of rank $r$ on $X$. We also get
$$H^1(T,\Upsilon)=H^1(X,\E)=H^0(X,K\otimes \E^*)^*=H^0(T_1,\Upsilon_{1}(-M))^{*};$$
here and in the sequel, we denote by $T_1$ the image of $\PP E_ 1$ by the linear system $\O(1)$.
For $X = \PP^1$, this becomes
\begin{multline*} H^1( \PP E , \Upsilon(-sF)) = H^1(\PP^{1}, \O (e_1-s) \oplus \O (e_2-s) \oplus \dots \oplus \O (e_r-s)) \\
 = H^0(\PP^1, \O (s-e_1-2) \oplus \O (s-e_2-2) \oplus \dots \oplus \O (s-e_r-2))^* =H^0(T_1,\Upsilon_1)^*. \end{multline*}
In fact $\Upsilon_1$ is $\O (1)$ on $\PP E_1$, where $\E_1=\O (s-e_d-2) \oplus \O (s-e_{d-1}-2) \oplus \dots \oplus \O (s-e_1-2)$ on $\PP^1$.
\par
The identifications of the $H^0$-spaces follows from \cite{Sc}, p.\ 110. Moreover, this and the the identification of the $H^1$-spaces follows from a straightforward generalization of Lemma V, 2.4 of \cite{H}: Clearly $H^i(\Upsilon(-M))_x = H^i(\Upsilon_x)=0$, for all $i>0$ and all points $x \in X$, since $\Upsilon)|_x=\O_{\PP^{d-1}}(1)$. Therefore $R^i(\pi_{*}\Upsilon(-M)) = 0$ for $i>0$. See \cite{H}, Chapter III, Ex.\ 11.8., and Chapter III, Ex.\ 8.4.

We note that $\PP E_1 \cong \PP E^*$, since $E_1 = E^* \otimes \Kx(D)$.
Hence the long exact cohomology sequence becomes:
\begin{multline} 0 \khpil H^0(\PP E,\Upsilon(-M)) \khpil H^0(\PP E,\Upsilon) \khpil (\Fq)^{sr} \\ \khpil H^0(\PP E^* , \Upsilon_1)^{*} \khpil  H^0(\PP E^* , (\Upsilon_1)-M)^{*} \khpil 0, \label{C} \end{multline}
which simplifies to
\begin{equation} 0 \khpil H^0(\PP E,\Upsilon) \khpil (\Fq)^{sr} \khpil H^0( \PP E^* , \Upsilon_1)^{*} \khpil 0 \label{CSAGS} \end{equation}
if $h^0( \PP E,\Upsilon(-M))=h^1( \PP E,\Upsilon)=0$. Dualizing, we get
\begin{multline} 0 \khpil H^0(\PP E^*,\Upsilon_1(-M)) \khpil H^0(\PP E^*,\Upsilon_1) \khpil (\Fq)^v \\ \khpil H^0( \PP E,\Upsilon)^{*} \khpil H^0( \PP E,\Upsilon(-M))^{*} \khpil 0 \label{Cdual} \end{multline}
which simplifies to
\begin{equation} 0 \khpil H^0(\PP E^* , \Upsilon_1) \khpil (\Fq)^{sr} \khpil H^0(\PP E , \Upsilon)^{*} \khpil 0 \label{CdualSAGS} \end{equation}
under the conditions stated. This motivates the following, generalizing the definition of a SAG code (see \cite{J}, $\S$2).
\begin{defn} A scroll code $C$ defined as above by evaluation of sections of $\Upsilon$ at exactly $r$ independent points of $s$ fibers is called a \textsl{strongly algebraic geometric scroll code} or \textsl{SAGS code} if $h^0( \PP E,\Upsilon(-M))=h^1( \PP E,\Upsilon)=0$. \end{defn}

The sequences (\ref{C}) and (\ref{CSAGS}) give that we obtain a generator matrix for $C$ by evaluating sections in $H^0(\PP E,\Upsilon)$ at the $v$ points. On the other hand, (\ref{Cdual}) and (\ref{CdualSAGS}) show that we get a generator matrix for (a code equivalent to) $C^*$, that is, a parity check matrix for $C$, by evaluating sections in $H^0(\PP E^*,\Upsilon_1)$ at some $v$ ``dual'' points (all of them in fibers corresponding to the same $s$ points over ${\bf P}^1$. We will say more about this in the next section.

\begin{rem} {\rm Recall that a Goppa code $C(D, G)$ is strongly algebraic geometric if $2g-2 < \deg G < s$. In analogy with this, we notice that $C$ is a SAGS code if $E$ is semistable and the following inequality holds: }
\begin{equation} r(2g-2) < deg(E) < rs. \label{SAGSineq} \end{equation}
{\rm For example, suppose $s \ge 2g$. By \cite{Na}, Remark 2.1, there exist semistable (in fact, even so-called $p$-semistable) bundles of degree zero and rank $r$ on $X$ for $X$, $r$ and $q$ ``general enough''. Twisting such a bundle by an effective divisor of degree strictly between $2g-2$ and $s$, we get an $E$ which defines a SAGS.} \end{rem}

\subsection{Another description of the codes}\label{altdesc}

Here we give another way of defining the codes $C$ and $C^*$ which will be useful for our work later with extensions.\\
\par
At each $x \in X$, a section $t$ of $O_{\PP E}(1) \to \PP E$ restricts to a linear form $t(x)$ on the projective space $\PP E|_x$; that is, a vector in $E|_x$. Evaluation of $t$ at $P = \langle e^{*} \rangle \in \PP E|_x$ is simply restriction of $t(x)$ to the line in $E^{*}|_x$ spanned by $e^*$. 
The points $P_{1,1} , \ldots , P_{s,r}$ come from covectors $e_{1,1}^{*} ,\ldots , e_{s,r}^{*} \in E^*$ which form a basis of each of the fibers of $E^*$ over the points of $D$. Thus there exist unique $e_{1,1} , \ldots , e_{s,r} \in E$ such that $e_{i,j}^{*}(e_{i^{\prime},j^{\prime}}) = \delta_{j, j^{\prime}}$, when this contraction makes sense (that is, when $x_{i^{\prime}} = x_i$). For each $(i,j)$, we have $\langle e_{i,j}^{*} \rangle^{*} = \langle e_{i,j} \rangle$, and restriction of $t(x)$ to $\langle e_{i,j}^{*} \rangle$ yields
\[ \hbox{(coefficient of $e_{i,j}$ in $t(x_{i})$)} \cdot e_{i,j} \]
which is well defined since the set of all the $e_{i,j}$ includes a basis of each of the chosen fibers. We write $\lambda_{i,j}$ for this coefficient. Identifying $\Fq^{sr}$ with $\bigoplus_{i,j} \Fq \cdot e_{i,j}$, we see that $t$ is sent to the $sr$-tuple $( \lambda_{1,1} , \ldots , \lambda_{s,r} )$. If we write this more suggestively as
\[ \left( ( \lambda_{1,1} , \ldots , \lambda_{1,r} ) , \ldots ( \lambda_{s-1,1} , \ldots , \lambda_{s,r} ) \right) \]
and consider $t$ now as a section of the vector bundle $E \to X$, then we see that the $r$-tuple $( \lambda_{(i,1} , \ldots , \lambda_{i,r} )$ is just the expression of $t(x_{i})$ in terms of our chosen basis of $E|_{x_i}$. We have natural identifications
\[ E|_{D} = \bigoplus_{i,j} \Fq \cdot e_{i,j} = \bigoplus_{i,j} O_{\PP E}(1)|_{\langle e_{i,j}^{*} \rangle} \]
allowing us to pass between the interpretations of $t$ as a section of $E \to X$ and of $\O_{\PP E}(1) \to \PP E$. Thus the sequence (\ref{CSAGS}) is identified with $0 \to H^{0} ( X , E ) \to E|_{D} \to H^{1} ( X , E(-D) ) \to 0$.\\
\par
We now set $H := E^{*}(D)$. Note that $H = \pi_{*} \left( O_{\PP E^*}(1) \otimes M \right)$. Now $E|_{D} = H^{*}(D)|_D$, which can be viewed as (the global sections of) the subsheaf of $\sPrin(H^{*})$ of principal parts supported at $D$ with at most simple poles. For each $(i,j)$, let $p_{i,j} \in \Prin(H^{*})$ be the principal part defined by $e_{i,j}$. (Of course, this is supported at $x_i$ with a simple pole.) Then we have $H^{*}(D)|_{D} = \bigoplus_{i,j} \Fq \cdot p_{i,j}$ and the sequence (\ref{CSAGS}) becomes
\[ 0 \to H^{0} ( X , H^{*}(D) ) \xrightarrow{\rho} H^{*}(D)|_{D} \xrightarrow{\nu} H^{1} ( X , H^{*} ) \to 0 \]
where $\rho$ and $\nu$ are induced by the principal part map\footnote{Since the poles are all simple, we could also think of this as the sum of the residue maps over the points of $D$.} and the coboundary map in (\ref{cohomseq}) respectively. Explicitly, $\rho$ sends a rational section of $H^{*}$ with poles bounded by $D$ to its principal part, and $\nu$ sends a principal part $\lambda_{1,1} p_{1,1} + \cdots + \lambda_{s,r} p_{s,r}$ to the cohomology class $\left[ \sum_{i,j} \lambda_{i,j} p_{i,j} \right]$.
\par
Thus the code $C$ is identified with the subspace of $H^{*}(D)|_D$ of elements occurring as principal parts of global rational sections of $H^*$, and the syndrome of an element in $\Fq^{sr}$ corresponds to the obstruction to lifting it to a global rational section of $H^*$.

\subsection{Generator and parity check matrices}\footnote{This subsection is logically independent of the rest.} In \cite{J}, generator and parity check matrices are given for the codes $C$ and $C^*$ when $r=1$, that is, $\PP E$ is the curve $X$. Here we generalize this approach to the present situation.
\par
Let $t_{1} , \ldots , t_l$ be a basis for $H^{0}(X, E)$. For each $m = 1 , \ldots , l$ and each $(i,j)$, write $\lambda_{m,(i,j)}$ for the coefficient of $e_{i,j}$ in $t_{m}(x_{i})$. Then by the last paragraph, the evaluation map sends $t_m$ to the principal part $\lambda_{m,(1,1)}p_{1,1} + \cdots + \lambda_{m,(s,r)}p_{s,r}$, so the matrix of $\rho$ with respect to the bases $\{ t_{m} \}$ and $\{ p_{i,j} \}$ is
\[ \begin{pmatrix} \lambda_{1,(1,1)} & \cdots & \lambda_{l,(1,1)} \\ \vdots & & \vdots \\ \lambda_{1,(s,r)} & \cdots & \lambda_{l,(s,r)} \end{pmatrix} =: S. \]
In order to find a matrix for $\nu$, in fact we will find one for $^{t}\nu \colon H^{1}(X, H^{*})^{*} \to \left( \Fq^{sr} \right)^*$ and dualize. We recall explicitly the Serre duality pairing
\[ H^{0}(X, \Kx \otimes H) \times H^{1}(X, H^{*}) \to H^{1}(X, \Kx) = \Fq. \]
Let $p$ be an $H^*$-valued principal part and $[ p ]$ its cohomology class; by (\ref{cohomseq}), every class in $H^{1}(X, H^{*})$ is of this form. Let $u$ be a global section of $\Kx \otimes H$. Then $u(p) \in \Prin(\Kx)$ and the contraction of $u$ and $[p]$ is simply $[ u(p) ]$. 
Hence $^{t}\delta(u)$ is the linear form given by $u \mapsto \left( p \mapsto [ u(p) ] \right)$.
\par
Now, for each $i$, let $z_i$ be a local coordinate on $X$ centered at $x_i$. We fix an isomorphism $\Fq \xrightarrow{\sim} H^{1}(X, \Kx)$ and let $c$ be the image of 1. We describe a basis of $\left( \Fq^{sr} \right)^*$ dual to the basis $p_{(1,1)} , \ldots , p_{(s,r)}$ of $\Fq^{sr}$. For each $(i,j)$, let $h_{i,j} \in H$ be such that $\langle h_{i,j} \rangle$ is the image of $\langle e_{i,j}^{*} \rangle$ under the natural isomorphism $\PP E = \PP (H^{*}(D)) \xrightarrow{\sim} \PP H^*$. We define a linear form $\overline{h_{i,j}}$ on $\Fq^{sr} = \bigoplus_{i,j} \Fq \cdot p_{i,j}$ by $p \mapsto \left[ dz_{i} \otimes h_{i,j}(p) \right]$. By construction, $\overline{h_{i,j}}(p_{i^{\prime},j^{\prime}})$ is nonzero if and only if $j = j^{\prime}$ and $i = i^{\prime}$. Multiplying the $h_{i,j}$ by nonzero scalars if necessary, we can assume that $\overline{h_{i,j}}(p_{i^{\prime},j^{\prime}}) = c \cdot \delta_{i,i^{\prime}} \delta_{j,j^{\prime}}$, so we obtain the required basis.
\par
Now let $u \in H^{0}( \Kx \otimes H)$. As we did for $E$ and $E^*$, for each $(i,j)$, let $h_{i,j}^*$ be the dual basis vector of $h_{i,j}$ in $H^*$. (Up to nonzero scalar, $h_{i,j}^{*} = e_{i,j}z_i$.) Then
\begin{equation} {^{t}\delta(u)} (p_{i,j}) = [ p_{i,j}(u) ] = \left[ \frac{u(x_{i})(h_{i,j}^{*})}{z_i} \right] = c \cdot \hbox{(coefficient of $dz_{i} \otimes h_{i,j}$ in $u(x_{i})$)} \label{hur} \end{equation}
Let us view $u$ as a section of the line bundle $\pi^{*}\Kx \otimes O_{\PP H}(1)$. To evaluate $u$ at the point $\langle h_{i,j}^{*} \rangle$, we restrict $u(x_{i}) \in (\Kx \otimes H)|_{x_i}$ to the line $\langle h_{i,j}^{*} \rangle$ in $H^{*}|_{x_{i,j}}$. This gives an element of $\Kx|_{x_i} \otimes \langle h_{i,j}^{*} \rangle^{*} = \Fq \cdot \left( dz_{i} \otimes h_{{i,j}} \right)$, and the coefficient is the same as that of $c$ in (\ref{hur}).
\par
Thus, if $u(x_{i}) = dz_{i} \otimes \left( \mu_{i,1} h_{i,1} + \cdots \mu_{i,r} h_{i,r} \right)$ for each $i$, then
\[ {^{t}\nu}(u) = \mu_{1,1} \overline{h_{1,1}} + \cdots \mu_{s,r} \overline{h_{s,r}} \]
is the expression of $^{t}\nu(u)$ with respect to the basis $\overline{h_{1,1}} , \ldots , \overline{h_{s,r}}$. 
Thus we can view $^{t}\nu(u)$ as the evaluation of $u$ at each of the points $\langle h_{i,j}^{*} \rangle \in \PP H$, expressed in terms of the $h_{i,j}$.
\par
Let now $u_{1} , \ldots , u_{l'}$ be a basis for $H^{0}(X, \Kx \otimes H)$. For each $n=1, \ldots , l'$, write $\mu_{n,(i,j)}$ for the coefficient of $dz_{i} \otimes h_{i,j}$ in $u_{n}(x_{i})$. Then the matrix of $^{t}\nu$ with respect to our chosen bases is
\[ \begin{pmatrix} \mu_{1,(1,1)} & \cdots & \mu_{l',(1,1)} \\ \vdots & & \vdots \\ \mu_{1,(s,r)} & \cdots & \mu_{l',(s,r)} \end{pmatrix} =: {^{t}R}. \]
The rows of this matrix give generators for $C^*$. But the $(ri+j)$th row represents the values of each of the $u_n$ at $\langle h_{i,j}^{*} \rangle$, so we see explicitly how $C^*$ is also an evaluation code. The matrix of $\delta$ with respect to $\{ p_{1,1} , \ldots , p_{s,r} \}$ and the basis of $H^{1}(X, H^{*})$ dual to $\{ u_{1}, \ldots , u_{l'} \}$ is $R$. By exactness, $RS = 0$ and $^{t}S {^{t}R}$ are zero, and $^{t}S$ and $R$ are parity check matrices for $C^*$ and $C$ respectively.\\
\\
\textbf{Note:} As we have defined them, $C$ and $C^*$ belong to different vector spaces. However, since we have the mutually dual bases $\{ p_{i,j} \}$ and $\{ \overline{h_{i,j}} \}$, we can view both codes as subspaces of $\Fq^{rs}$ via the vector space isomorphism $\Fq^{rs} \xrightarrow{\sim} \left( \Fq^{rs} \right)^*$ sending each $\overline{h_{i,j}}$ to $p_{i,j}$.

\begin{rem} \label{dual}
{\rm It follows from the discussion above that the orthogonal complements
(or duals) of SAGS codes are (code equivalent to) SAGS codes in general, just like for the traditional case $r=1$. Hence the description above lends itself just as well to make parity check matrices as to make generator matrices.
This is one of the virtues of Goppa codes (based on curves), which it has been hard to reproduce for codes produced from varieties of higher dimension.
If one picks all points of for example Grassmannians or  scrolls, then the coordinates of these points are suitable for producing columns of generator matrices
of codes that are interesting. But if one tries to use the same points as columns of parity check matrices, then because of the existence of linear spaces inside the varieties (lines), one cannot exceed minimum distance $3$. Hence, in order to get essentially self-dual classes of codes, like for Goppa codes, one must revise the way one picks points.}
\end{rem}

\subsection{The link with extensions}

Since the column vectors of the parity check matrix of $C$ are described through coordinates of points of $\PP E^*$ embedded by the complete linear system $\Upsilon_1$, we see that the (projectivized) syndrome space of $C$ in a natural way is identified with
\[ \PP H^0(\PP E^* , \Upsilon_1) = \PP H^0(X, K(D) \otimes E^*). \]
If $X=\PP^1$ and $\Upsilon=\O(e_1)\oplus \cdots \O(e_d)$, then this is 
\[ \PP H^0(\PP E^* , \Upsilon_1) = \PP H^0(\PP^1, \O (s-e_1-2) \oplus \O (s-e_2-2) \oplus \dots \oplus \O (s-e_d-2)) \]
The syndrome space can also be identified with
$$ H^1(X,E \otimes M^*)=H^1(X,H^*),$$ where as before $H^{*} = E(-D)$. For $X = \PP^1$ and $\Upsilon =\O(e_1)\oplus \cdots \O(e_d)$, this is
$$ H^1(\PP^1, \O (e_1-s)) \oplus \O (e_2-s)) \oplus \dots \oplus \O (e_d-s))=H^1(X,H^*),$$
where $H=\O (s-e_1) \oplus \O (s-e_2) \oplus \dots \oplus \O (s-e_d)$.
\par
Now $H^1({X},H^*)= Ext^1(\O_{{X}},H^*)$ can be identified with isomorphism classes of extensions 
     $$ 0 \khpil H^* \khpil W \khpil \O_{{X}} \khpil 0.$$ 
In the next section, we will relate the geometry of the space $\PP H^1({X},H^*)^*$ to the behaviour of these extensions.

\section{Geometry of extension spaces} \label{geomext}

Henceforth, to allow for slightly greater generality, instead of the bundle $H^*$ we will work with $\Hom(F_{2}, F_{1})$ for bundles $F_1$ and $F_2$ over $X$. Recall that the \textit{decomposable locus} of $\Hom(F_{2}, F_{1})$ is the locus of maps of rank one. This is a determinantal subvariety of $\Hom(F_{2}, F_{1})$, defined by the vanishing of all $2 \times 2$ minors of the maps. We denote $\Delta$ the locus defined by these (homogeneous) polynomials in $\PP \Hom(F_{2}, F_{1})^*$.

\begin{exa} {\rm If $F_1$ and $F_2$ are both of rank two then $\Delta$ is a bundle of smooth quadrics in the $\PP^3$-bundle $\PP \Hom(F_{2}, F_{1})^{*} \to X$. Of course, if either one is a line bundle then $\Delta = \PP \Hom(F_{2}, F_{1})^*$.} \end{exa}  

\subsection{Embeddings of scrolls}

Here we give another description of the map from $\PP H$ into the projectivized syndrome space $\PP H^{1}(X, H^{*})^*$, which will be adapted for our study of extensions.
\par
Let $V \to X$ be any vector bundle. We have a short exact sequence
\[ 0 \to \V \to \V(x) \to \frac{\V(x)}{\V} \to 0 \]
whose cohomology sequence includes
\[ \cdots \to H^{0}(X, V(x)) \to V(x)|_{x} \to H^{1}(X, V) \to \cdots \]
Since $\PP ( V^{*}(-x) )|_x$ is canonically isomorphic to $\PP V^{*}|_x$, the projectivized coboundary map gives rise to a map $\psi_{x} \colon \PP V^{*}|_{x} \dashrightarrow \PP H^{1}(X, V)^*$. We define a map $\psi \colon \PP V^{*} \dashrightarrow \PP H^{1}(X, V)^*$ by taking the product of all the $\psi_x$.
\par
Now by Serre duality and the projection formula, we have an identification
\begin{equation} H^{1}(X, V) \xrightarrow{\sim} H^{0}(\PP V^{*} , \pi^{*}\Kx \otimes O_{\PP V^*}(1))^{*}. \label{Serreproj} \end{equation}
\begin{lemma} (\cite{Hi}, $\S$2) Via the above identification, $\psi$ coincides with the standard map $\PP V^{*} \dashrightarrow |\pi^{*}\Kx \otimes O_{\PP V^*}(1)|$. Moreover, $\psi$ is an embedding if and only if for all $x, y \in X$, we have $h^{0}(X, \Kx(-x-y) \otimes V^{*}) = h^{0}(X, \Kx \otimes V^{*}) - 2r$. \label{psiprops} \end{lemma}

\begin{rem} {\rm The key feature of this interpretation is that $\psi$ sends $\langle v \rangle \in \PP V^{*}|_x$ to the projectivized cohomology class of a $V$-valued principal part supported at $x$ with a simple pole in the direction $v$. This will allow us to use the alternative construction of the code $C$ in $\S$\ref{altdesc} to understand the geometry of the syndromes.} \end{rem}

We recall a definition:

\begin{defn} Let $F_{2} \to X$ be a vector bundle. Then an \textsl{elementary transformation of $F_2$} is a vector bundle defined by a locally free subsheaf of $\F_2$ of rank equal to the rank of $F_2$. \end{defn}

Such subsheaves can be defined using principal parts. If $F_1$ is another vector bundle over $X$, then any $\Hom(F_{2}, F_{1})$-valued principal part naturally defines a map $\F_{2} \to \sPrin(F_{1})$. Then the kernel of such a map defines an elementary transformation of $F_2$. Moreover, any elementary transformation of $F_2$ is of this form (although not in a unique way).\\
\par
We will need the following technical result on extension classes:
\begin{lemma} (\cite{Hi}, $\S$4.1) Let $W$ be an extension of $F_2$ by $F_1$. An elementary transformation of ${\mathcal G}$ of ${\mathcal F}_2$ lifts to a vector subbundle of $W$ if and only if the class $\delta(W)$ of the extension can be defined (cf.\ (\ref{cohomseq})) by a principal part $p \in \Prin(\Hom(F_{2}, F_{1}))$ such that ${\mathcal G} = \ker \left( p \colon {\mathcal F}_{2} \to \sPrin(F_{1}) \right)$. \label{tech} \end{lemma}

Now we can give the main result of this section.

\begin{thm} Let $0 \to F_{1} \to W \to F_{2} \to 0$ be a nontrivial extension. Then $\langle \delta(W) \rangle$ belongs to the linear span of at most $h$ independent points of $\Delta|_D$ if and only if $W$ has a subbundle lifting from an elementary transformation of $F_2$ of the form
\begin{equation} 0 \to {\mathcal G} \to {\mathcal F}_{2} \to \tau \to 0 \label{set} \end{equation}
where $\tau \subset \sPrin(F_{1})$ is a skyscraper sheaf of length at most $h$ supported on $D$ and with at most simple poles. \label{geomlift} \end{thm}
\begin{proof} Suppose $\langle \delta(W) \rangle$ belongs to the linear span of at most $h$ independent points of $\Delta|_D$ in $\PP H^{1}(X, H^{*})^*$. Then by the alternative definition of $\psi$ given above, $\delta(W)$ can be defined by $p \in \Prin(\Hom(F_{2} , F_{2})$ of the form $\sum_{j=1}^{h}p_j$ where each $p_j$ is a principal part supported at one point of $D$ with a simple pole along some rank one map. Thus we have a short exact sequence
\[ 0 \to \G \to \F_{2} \xrightarrow{p} \tau \to 0 \]
where $\tau \subset \sPrin(F_{1})$ is a skyscraper sheaf of length at most $h$ supported on $D$ and with at worst simple poles. By Lemma \ref{tech}, the sheaf $\G$ lifts to a subbundle of $W$.\\
\par
Conversely, suppose an elementary transformation $\G$ of $\F_2$ of the stated type lifts to a subbundle of $W$. By Lemma \ref{tech}, the class $\delta(W)$ can be defined by some $p \in \Prin(\Hom(F_{2}, F_{1}))$ which, viewed as a map $\F_{2} \to \sPrin(F_{1})$, has kernel $\G$. From the sequence (\ref{set}) we deduce that $p$ is supported along $D$ and has at most simple poles. We write $p = \sum_{x \in D}p_x$, where each $p_x$ is a principal part supported at one point $x$, and then write each $p_x$ as a sum of rank one homomorphisms, of minimal length. Since $\tau$ is of length at most $h$, any such expression for $p$ contains at most $h$ independent such rank one homomorphisms. By the alternative definition of $\psi$, the point $\langle \delta(W) \rangle$ is contained in the span of these at most $h$ rank one points of $\Hom(F_{2}, F_{1})$. \end{proof}

\begin{rem} {\rm Suppose $F_{1} = H^*$ and $F_{2} = \Ox$, so we are in the situation of the last section. Then this theorem shows that the extension $0 \to H^* \to W \to \Ox \to 0$ can be ``quasi-inverted'' to a short exact sequence
\[ 0 \to \Ox(-A) \to W \to (H^{*})^{\prime} \to 0 \]
for some effective divisor $A \leq D$ and some bundle $(H^{*})^{\prime}$, if and only if $\langle \delta(W) \rangle$ lies in the linear span of some points of $\PP H^*$ all lying over the support of the divisor $A \leq D$. In this case the rank of each $\phi_x$ can be at most 1.} \end{rem}

\begin{rem} {\rm When $\Fq$ is replaced with the complex number field, a generalization of Theorem \ref{geomlift} is proven in \cite{Hi}, Theorem 4. For example, $p$ may have poles of higher order. However, the above proof suffices for the situation we are considering.} \end{rem}

Now we give some alternative ways of viewing the condition of Theorem \ref{geomlift}. Firstly, it is equivalent to saying that $\delta(W)$ belongs to
\begin{multline*} \ker (H^1(X,\Hom(\O_{X},H^*)) \khpil H^1(X,\Hom(\O_{X}(-A),H^*))) \\
= \ker(H^1(X,H^*) \khpil H^1(X,H^*(A)) \\
= \ker(H^0(X,H(K))^* \khpil H^0(X,H(K-A))^*). \end{multline*}
We look instead at the (isomorphic) space of extensions of type
      $$ 0 \khpil \O_X  \khpil W \khpil H \khpil 0.$$
Then it follows from a dual version, as in \cite{NR}, Lemma 3.2, that there is a 
surjection $W \khpil \O_X(A) \khpil 0$ if and only if $\delta(W)$ 
belongs to
\begin{multline*} \ker (H^1(X,\Hom(H,\O_X) \khpil H^1(X, \Hom(H,\O_X(A))) \\
= \ker(H^1(X,H^*) \khpil H^1(X,H^*(A)). \end{multline*}
We see that the two kernels are the same, and we have two alternative descriptions.
\par
In the first description we may view  $\O_X(-A)$ as a special case
of a locally free sheaf $G$ with a sheaf injection $\phi: G \khpil \O_X,$ 
such that $\phi$ factors via a map $G \khpil W.$
\par
In the second description  $\O_X(A)$ is a special case of 
a locally free sheaf $G$ with a sheaf homomorphism 
$\phi: \O_X \khpil G,$ such that $\phi$ 
extends to a homomorphism $W \khpil G$. 

The common kernel can also be viewed as that of a map
     \[Ext(\O_X, H^*) \khpil Ext(\O_X, H^*(A)),\]
using Proposition 6.3 of \cite{H}. Using Proposition 6.7 of \cite{H},
we interpret this as the kernel of a map
     \[Ext(H,\O_X) \khpil Ext(H(-A),\O_X)=Ext(H,\O_X(A)).\]
\begin{example} \label{dir}
{\rm An easy case to handle is when $X= \PP^1$ and we pick all $r$ points in each fiber along the directrix curves. Assume $s=q+1$. It is clear that the natural subscroll of type $(e_2, \dots, e_r)$ is contained in a hyperplane, and that any hyperplane containing this subscroll intersects the first directrix in $e_1$ points, and for some such hyperplane they can taken to be rational. One easily sees that this hyperplane contains $e_1+(r-1)(q+1)$ points, which is largest possible, and that the minimum distance of the code then is $q+1-e_1$. Working dually, with the scroll $T_1$ of type $(q-1-e_r, \dots q-1-e_1)$, we see that we get a linear dependency between $q+1-e_1$ points on the directrix curve of smallest degree $(q-1-e_1)$, which again indicates minimum distance $q+1-e_1$. Furthermore the higher weights $d_i$ increase by one until we reach a value of $i$ such that no codimension $i$-space contains the subscroll of $(e_2, \dots, e_r)$. This space contains $e_1+e_2+(r-2)(q+1)$ points, so $d_i=2q+2-e_1-e_2$. We leave it to the reader to make the remaining calculations to the determine the complete weight hierarchy. It is a sad fact that a code with such a nice description has such bad code-theoretical properties.}
\end{example}

\section{Error correction}
\label{geomcorr}

We return to the code $C$. Here we give a geometric condition for the correctability of the error, in terms of the image of the embedding of $\PP \Hom(F_{2}, F_{1})^*$ in the syndrome space.\\
\\
\textbf{Important hypothesis:} We will assume that the $\Hom(F_{2}, F_{1})$-valued principal parts $p_{1,1}, \ldots , p_{s,r}$ are all along directions corresponding to rank one homomorphisms. This is possible since each fiber is spanned by such maps.

\subsection{A geometric condition for correctability}

The following is analogous to \cite{J}, Theorem 3.4.

\begin{thm} Suppose a codeword $\x \in C$ is transmitted, and $\y = \x + \e$ is received. Let $W$ be the extension of $F_2$ by $F_1$ defined by the syndrome class $\nu(\y) = \nu(\e)$. Then the error $\e$ has weight at most $h$ only if $W$ has a subbundle, necessarily of degree at least $\deg(F_{2}) - h$, lifting from an elementary transformation of $F$ of the form
\[ 0 \to {\mathcal G} \to {\mathcal F}_{2} \to \tau \to 0 \]
where $\tau \subset \sPrin(F_{1})$ is a skyscraper sheaf of length at most $h$ supported on $D$ and with at most simple poles. \label{corrgeom} \end{thm}
\begin{proof} Suppose $\e$ has weight $h$. Then $\e$ is a principal part of the form
\[ \lambda_{i_{1}, j_{1}} p_{i_{1}, j_{1}} + \cdots + \lambda_{i_{h},j_{h}} p_{i_{h}, j_{h}} \]
for some nonzero $\lambda_{i_{1},j_{1}}, \ldots , \lambda_{i_{h},j_{h}} \in \Fq$, with the $(i_{h}, j_{h})$ all distinct. Then by Lemma \ref{tech}, the kernel of the map of $\strs$-modules $\e \colon \F_{2} \to \sPrin(F_{1})$ lifts to a subbundle of $W$. Since $\e$ is supported along $D$ and is a sum of $h$ rank one elements of $\Hom(F_{2}, F_{1})$ with at most simple poles along $D$, this subbundle is an elementary transformation of the stated type. \end{proof}

\begin{rem} {\rm As in Theorem 3.4 of \cite{J}, we can give a geometric interpretation of this situation. By Theorem \ref{geomlift}, there are at most $h$ errors in $\y$ only if the class $\langle \nu(\e) \rangle$ belongs to an $h$-secant plane to $\Delta|_D$ spanned by at most $h$ distinct points. Moreover, also as in \cite{J}, both the errors and the points of the scroll $\PP H^*$ 
are defined over $\Fq$.} \end{rem}

\subsection{Error location}

Assume we have a code $C$ as described above, and that we have picked exactly $r$ points in each of $s$ fibers, where $s$ typically is the number of all $\Fq$-rational points on $X$. Assume a codeword is sent, and the syndrome calculated. If this is zero, there is no problem. Otherwise, look at the corresponding point in the projectivized syndrome space $\PP Ext^1(\O_{{X}},H^*)^* = \PP Ext^1(H,\O_{X})^*$. This is the space where the scroll $\PP H \cong \PP E^{*} \cong T_1$ is embedded.
\par
Now we think of error location in two steps. In Step 1 we find the various fibers of $T_1$ such that the syndrome is a linear combination of points of these fibers. In Step 2 we find the individual points in these fibers such that the syndrome contribution from each given fiber is a linear combination of syndromes from these individual points. It is only  in Step 1 that we can use the direct analogue with the situation studied in \cite{BC}, \cite{Co1}, \cite{Co2} and \cite{J}. On the other hand, Step 2 is basically only a linear algebra problem: the syndrome component from this fiber is a point of this fiber. Find which linear combination it is, of the $r$ (dual) points that we have picked in this fiber in the first place.
\par
Hence we focus on Step 1. View the syndrome as an extension
   $$ 0 \khpil \O_X \khpil W \khpil H \khpil 0.$$ 
Error location is then to find the (hopefully) unique line bundle $\O_X(A)$ of lowest degree such that there is a surjection of $W$ onto $\O_X(A)$. Having found the divisor class, one must find the effective divisor $A'$ in the class such that the syndrome is spanned by the fibers of $T_1$ corresponding to points on the divisor $A'$. This part of the process is described in Chapter 4 of \cite{BC} for $r=1$.

\begin{defn}
For a rank $r$ bundle $V$ on a curve $X$ we set
$s_1(V)= \deg V - r\max \{\deg L\},$ for $L$ a line subbundle of 
$V$ on $X$.
\end{defn}
Then we have:
\begin{prop} \label{main}
For a given syndrome point $\nu(\y)$, interpreted as an extension of type
$$ 0 \khpil H^* \khpil W \khpil \O_X \khpil 0,$$ we have 
$$s_1(W) \le (r+1)a-rs+\deg \E,$$
where $a$ is the number of different
fibers of $T$ (over $X$) we must use to pick points such that errors in the 
positions corresponding to these points give rise to the syndrome.
\end{prop}
\begin{proof}
If we use points from $a$ different fibers to span the syndrome, where these
fibers correspond to the points $x_{i_1}, \cdots x_{i_a}$ on the curve, then
$\mathcal A^*$ is a subsheaf of rank one of $\mathcal W$, and 
$$s_1(W) \le \deg W + \deg A= \deg H^*+(r+1)a=\deg \E -rs +(r+1)a.$$
\end{proof}
\begin{rem}
{\rm For $r=1$ and Goppa codes this gives $s(W) \le 2a -\deg (D-G)=2a-d$,
where $d$ is the designed minimum distance, and we see that if the number of errors is less than designed-correctable, then $W$ is unstable. 
}

\end{rem}
\par
What does it take to ensure that the fibers spanning a point can be uniquely chosen? Assume there are errors in two fibers. If the syndrome point $\langle \nu( \e ) \rangle$ is also in the span of two other fibers, then the span of the first two fibers has a common point with the span of the second group of two fibers, and the span of all four fibers is less than the ``expected'' value which is $4r-1$. Hence a sufficient condition for this not to happen is
\[ h^0( T_1 , \Upsilon_{1}(- F_{1} - F_{2} - F_{3} - F_{4})) = h^0( X, E_1( - x_{1} - x_{2} - x_{3} - x_{4})) = h^0(X, E_1) - 4r \]
for all choices of $x_1, x_2, x_3, x_4$.
\par
In general, syndromes from $a$ fibers can be uniquely traced back to $a$ fibers if
\[ h^0(X, E_1(-B))=h^0(X, E_1)-2ar \]
for all choices of effective divisors $B$ of degree $2a$ (compare with Lemma \ref{psiprops}).
\par
For $X=\PP^1$ this happens if $s-e_1-2-2a \ge -1$. For curves of higher genus we have:
\begin{prop}\label{stab}
Let $E$ be a stable bundle of rank $r$ on $X$. Then errors in
\[ a<\frac{\mu (H)}{2}= \frac{s-\mu(\E)}{2} = \frac{rs-\deg \E}{2r} \]
different fibers can be traced back to a unique choice of $a$ fibers.
\end{prop}
\begin{proof}
By Riemann--Roch, we have
\[ h^0(X, E_{1}(-B)) = \deg E_1 - 2ra + r(1-g) + h^0(\Kx(B) \otimes E_{1}^{*}). \]
We show that $h^0(X, \Kx(B) \otimes E_{1}^{*}) = h^0(X, H^{*}(B))=0$. Since $E$ is stable, so is $H^*$, and one obtains $h^0(X, H^{*}(B))=0$ unless $2a \ge \mu(H)$. (By the same argument, $H^0(X, E_1)=\deg E_1 +r(1-g)$, so we obtain the desired conclusion.) \end{proof}
\begin{rem}
{\rm Since one can correct errors from $a$ fibers if $a \le \frac{\mu (H)-1}{2},$
it is tempting to conclude that one can correct up to $t=r \left( \frac{\mu (H)-1}{2} \right) = \frac{\deg H-r}{2}$ errors (which holds for $r=1$, where $\deg H$ is the designed minimum distance), since there are $r$ points in each fiber. But the discussions so far  only makes this true for $t$ errors if they are clustered in as few as $\frac{\mu (H)-1}{2},$ fibers. If they are ``spread out'' on more fibers, the discussion above does not ensure unique decoding of more than $\frac{\mu(H)-1}{2}$ errors.}
\end{rem}

\begin{prop} Suppose $E$ is a stable bundle on $X$. Then the syndrome of an error which can be traced uniquely back to a choice of $a$ fibers in the sense of the previous result, that is, where the number $a$ of fibers where errors are made is at most
$\frac{\mu (H)}{2}$, defines an extension $0 \to H^{*} \to W \to \O_X \to 0$ with
$$s_1(W) < \frac{(r-1)(\deg \E - sr)}{2r}.$$
\end{prop}
\begin{proof}
Insert $a=\frac{\mu (H)}{2}=\frac{rs- \deg \E}{2r}$ in the statement of
Proposition \ref{main}.
\end{proof}
We see that the right hand side is strictly negative if $r>1$ and $C$ is a SAGS code.
\begin{rem} \label{perf}
{\rm What can be said about the code parameters of the SAGS codes?
Certainly the word length is $sr$. If enough fibers are chosen ($s$ big enough) that the chosen fibers span the projective space $\mathbb{P}^{k-1}$ in which $T$ lies, then the dimension of the code is $k$. It is much harder to find the true minimum distance of the codes. All we can say is that it depends on the choice of the points in each fiber. The case studied in Example \ref{dir} obviously 
represents bad choices. To illustrate the problem of choosing points conveniently, look at the simplest case of a code which is not a Goppa code.
We choose $X=\PP^1$, and $L=\O(1) \oplus \O(1)$. Hence $T$ is a quadric
in  $X=\PP^3$. How can we choose $2$ points on each line on one of the 
two families, such that as few as possible among the $2q+2$ points are 
contained in the same plane. The worst case involves $q+2$ points in a plane
(take two lines $L_1$ and $L_2$ on the quadric, meeting in a point $P$), and choose all $q+1$ points on $L_2$, one additional point on $L_1$, and q additional points on lines parallel to $L_1$).
But there clearly exist better choices, unless $q$ is very small.
In general, the minimum distance guaranteed by the method, is $s-2-e_1$
for codes from rational curves (Example \ref{dir}) and $\mu (H)$ for codes from curves of higher genus if $E$ is stable (Proposition \ref{stab}).
But one can hope for much better true values with good choices of points. }
\end{rem}
\begin{rem}
{\rm The considerations above involve no direct decoding algorithm. Neither did \cite{J}. Nevertheless the principles from \cite{J} were used to approach concrete decoding in the series of papers \cite{BC}, \cite{Co1}, \cite{Co2}. We feel that the generalization presented in the present paper of the line of thoughts in \cite{J}, should lend itself to a corresponding generalization of the results of these other papers.}
\end{rem}

\section{Syndrome decoding of other codes}
\label{syn2}
Throughout this work we have insisted on picking exactly $r$ points in each fiber when defining the codes. This is because we have defined the codes as evaluation codes, thus starting with a natural generator, and not a parity check matrix. To obtain the exact sequence
\[ 0 \khpil H^0( \PP E, \Upsilon ) \khpil (\Fq)^{sr} \khpil H^0( \PP E^{*}, \Upsilon_1)^{*} \khpil 0\]
and its dual counterpart
\[ 0 \khpil H^0( \PP E^{*} , \Upsilon_1) \khpil (\Fq)^sr \khpil H^0( \PP E , \Upsilon)^{*} \khpil 0,\]
we had to choose $r$ points in each fiber, spanning it. As in $\S$\ref{matrix}, the last sequence defines a parity check matrix via evaluation of the dual points in each fiber.
\par
An easier approach for our purposes would of course have been to start with $T_1$ and its complete linear system $\Upsilon_1$ in the first place, and define a code $C$ by a parity check matrix obtained from evaluation of sections of $\Upsilon_1$ in some more arbitrarily chosen points on the fibers of $T_1$. As an extreme case we could have picked all $\Fq$-rational points of all fibers. Everything said above about Step 1 of the last section would then have been unaltered, but in Step 2 we would be far from having unique decoding, unless we knew for some reason that at most a single error could be made in each individual fiber. (As mentioned in Remark \ref{dual}, the minimum distance would be $3$ in this extreme case).
\par
Nevertheless, if one defines codes from scrolls via parity check matrices instead of via generator matrices (as evaluation codes), then one obtains a larger class of codes, for which one can interpret decoding as described above via vector bundle manipulations.

\vspace{.5cm}

\end{document}